\documentclass[12pt,thmsa, reqno]{amsart}

\usepackage{mathrsfs}
\usepackage{amsfonts}
\usepackage{amssymb}
\usepackage{amsmath}

\usepackage[dvips]{graphics}
\input{amssym.def}
\input{amssym.tex}

\def\th{\theta}

\def\Cal{\mathcal}

\def\N{{\Cal N}}

\def\I{{\Cal I}}

\def\T{{\Cal T}}

\def\f0{f_0}
\def\Fc0{\varphi_0}
\def\rn{\bbr^n}

\def\I_k {I_{-}^{k/2}}
\def\I+k {I_{+}^{k/2}}

\def\bbr{{\Bbb R}}

\def\Lip{{\hbox{\rm Lip}}}

\def\part{\partial}
\def\intl{\int\limits}
\def\b{\beta}

\def\Gam{\Gamma}
\def\Om{\Omega}
\def\a{\alpha}
\def\om{\omega}

\def\Del{\Delta}
\def\del{\delta}
\def\vp{\varphi}

\def\gam{\gamma}

\def\sig{\sigma}
\def\lam{\lambda}
\def\z{\zeta}
\def\e{\varepsilon}
\def\t{\tau}

\font\frak=eufm10

\def\fr#1{\hbox{\frak #1}}

\def\frC{\fr{C}}

\def\part{\partial}
\def\intl{\int\limits}
\def\b{\beta}

\def\Gam{\Gamma}
\def\Om{\Omega}
\def\a{\alpha}

\newtheorem{theorem}{Theorem}[section]
\newtheorem{lemma}[theorem]{Lemma}

\newtheorem{corollary}[theorem]{Corollary}

\theoremstyle{remark}
\newtheorem{remark}[theorem]{Remark}

\newtheorem{example}[theorem]{Example}

\numberwithin{equation}{section}

\newcommand{\be}{\begin{equation}}
\newcommand{\ee}{\end{equation}}

\newcommand{\bea}{\begin{eqnarray}}
\newcommand{\eea}{\end{eqnarray}}
\newcommand{\Bea}{\begin{eqnarray*}}
\newcommand{\Eea}{\end{eqnarray*}}

\def\sideremark#1{\ifvmode\leavevmode\fi\vadjust{\vbox to0pt{\vss
 \hbox to 0pt{\hskip\hsize\hskip1em
\vbox{\hsize2cm\tiny\raggedright\pretolerance10000
 \noindent #1\hfill}\hss}\vbox to8pt{\vfil}\vss}}}%

\begin{document}

\title[The Vertical Slice Transform]
{The Vertical Slice Transform in Spherical Tomography}

\author{B. Rubin }
\address{Department of Mathematics, Louisiana State University, Baton Rouge,
Louisiana 70803, USA}
\email{borisr@lsu.edu}






\begin{abstract}

The  vertical slice transform  takes  a function  on the unit sphere  in $\bbr^{n+1}$, $n\ge 2$,   to  integrals of that function over spherical slices parallel to the last coordinate axis.  In the case $n=2$ these transforms  arise in thermoacoustic tomography.  We obtain  new inversion formulas for the vertical slice transform and  its singular value decomposition.  The results can be applied to the inverse problem for the Euler-Poisson-Darboux equation associated to the corresponding spherical means.
 \end{abstract}

\maketitle

\section{Introduction}
\setcounter{equation}{0}

Let $\T$  be the set of all cross-sections  of the unit sphere $S^n$ in  $\bbr^{n+1}$ by the hyperplanes parallel to the last coordinate axis.
 The {\it vertical slice transform} takes a function $f$ on $S^n$ to a function $Vf$ on  $\T$ by integration over these cross-sections. Specifically,
 \be\label{vt}  (Vf) (\t)=\intl_\t f(x)\, d_\t x, \qquad \t \in \T,\ee
where $d_\t x$ stands for the induced surface measure on $\t$.

Transformations (\ref{vt}) with $n=2$ arise in thermoacoustic tomography; see  \cite{GRS, HQ,  ZS}, where
 inversion of $Vf$ is reduced to the classical Radon transform for lines in the plane \cite{GGG, H}. In Section 2 we extend this approach to all $n\ge 2$, derive  new inversion formulas and investigate the corresponding  singular value decompositions.

An alternative way to study integrals (\ref{vt}) is to treat them  as the spherical means over geodesic spheres centered on the equator of $S^n$. This situation resembles the  Euclidean case  having wide application in photoacoustic and thermoacoustic tomography, when the  spherical means are evaluated over  spheres centered on a boundary of a ball; see, e.g., \cite { AKQ,  AN, AK, AER, FHR, FR,  Ku14, KK08, KK15, Ku, M13,  N, P, XW}, to mention a few. In Section 3 we will show that the operator (\ref{vt}) can  be explicitly inverted using the method of analytic continuation developed in our previous papers \cite{Ru08b, AER}. This  alternative approach might be of independent interest.

\vskip 0.1 truecm

\noindent{\bf Notation.}

 We write $x\in \bbr^{n+1}$ as $(x', x_{n+1})$, where $x'\in \rn$, $n\ge 2$.  In the following,   $S^n$ is the unit sphere in  $\bbr^{n+1}$,  $S^n_+= \{x \in S^n: x_{n+1}\ge 0\}$ is the ``upper" hemisphere,  $B_n=\{x'\in \rn: |x'|\le 1\}$ is the unit ball in $\rn$; $S^{n-1}$ is the  boundary of $B_n$; $\sig_{n-1}\!=\!2\pi^{n/2}/\Gam (n/2)$ is the area of  $S^{n-1}$.
 If $x\in S^n$, then $d\sig(x)$ stands for the Riemannian surface measure on $S^n$. The standard notation $C(\Om)$ and $L^p(\Om)$ is used for the corresponding spaces of continuous and $L^p$ functions on the set $\Om$ under consideration.

 Everywhere in the following, we assume that the function $f$ in (\ref{vt}) is  even  in the last variable, because if $f(x', \cdot)$ is odd, the integral (\ref{vt}) is zero.

\section{Connection with the Hyperplane Radon Transform}

The following  lemma is a standard fact from Calculus.
\begin{lemma} \label{lem2.1} If $f\in L^1 (S^n_+)$, then
\be\label{vt2} \intl_{S^n_+} f(x)\,d\sig(x)=\intl_{B_n} \vp (x')\, dx', \qquad \vp (x')=\frac{f(x', \sqrt{1-|x'|^2})}{\sqrt{1-|x'|^2}}.\ee
\end{lemma}
\begin{proof}
\bea
\intl_{S^n_+}\! f(x) \, d\sig (x)\!\!&\!=\!\!&\!\intl_{B_n} f(x',  \sqrt{1\!-\!|x'|^2}) \sqrt{1\!+\!
\left (\frac{\partial x_{n+1}}{\partial x_1}\right )^2\!+\!\ldots \!+\!\left (\frac{\partial x_{n+1}}{\partial x_n}\right )^2} dx'\nonumber\\
&\!=\!&\!\!\intl_{B_n} f(x',  \sqrt{1\!-\!|x'|^2}) \,\frac{dx'}{ \sqrt{1\!-\!|x'|^2}}=\intl_{B_n} \vp(x')\, dx'. \nonumber\eea
\end{proof}

Every vertical hyperplane section $\t$ of $S^n$ can be parametrized by the pair $(\theta, t)$, where $\theta \in S^{n-1}$ and $t\in (-1,1)$ is the signed distance from the origin to the hyperplane containing $\t$. We denote
\[\frC_n=\{(\th,t): \th \in S^{n-1},  \,t\in (-1,1)\}.\]
Owing to the evenness of $f(x', \cdot)$,
\be\label{vt1a}
(Vf)(\t)\equiv (Vf)(\theta, t)=2 \intl_{S^n_+ \cap \,\t} f(x)\, d_{\t}\sig (x),\ee
where $d_{\t}\sig (x)$ stands for the surface element on $\t$. Thus, without loss of generality, we can restrict ourselves to the hemispherical transform
\be\label{vt1a}
(V_+f)(\t)\equiv (V_+f)(\theta, t)=\intl_{S^n_+ \cap \,\t} f(x)\, d_{\t}\sig (x)\ee
that might be of independent interest.

Let
\be\label{vt3}
(R\vp) (\theta, t)\!\equiv \!(R_\theta\vp) (t)\!=\!\!\intl_{x' \cdot \theta =t} \!\!\vp (x')\, d_{\theta,t} x', \qquad \theta\in S^{n-1}, \; t\in \bbr,\ee
be the hyperplane Radon transform of the function $\vp$ from (\ref{vt2}), where $d_{\theta,t} x'$ is the volume element of the hyperplane $x' \cdot \theta =t$.
Clearly, $(R\vp) (\theta, t)=0$ whenever $|t|\ge 1$.

\begin{lemma} \label{lem2.2} If $f\in L^1 (S^n_+)$, then for almost all $(\theta, t) \in \frC_n$,
\be\label{vt4}  (V_+f)(\theta, t)=\sqrt {1-t^2}\, (R\vp) (\theta, t).\ee
 If $f\in C(S^n_+)$, then (\ref{vt4}) holds for all $(\theta, t) \in \frC_n$.
\end{lemma}
\begin{proof} By rotation invariance, it suffices to prove the lemma for $\theta=e_n=(0, \cdots, 1,0)$, when $x'=te_n +y$, $\, y\in \bbr^{n-1}$, $\, |y|< r=\sqrt{1-t^2}$. As in the proof of Lemma \ref {lem2.1},
\bea
(V_+f)(e_n, t)\!\!&=& \!\!\intl_{|y|< r}\!\! f\left(te_n \!+\!y, \sqrt{r^2\!-\!|y|^2}\right)  \sqrt{1+\sum\limits_{j=1}^{n-1} \left (\frac{\partial \left(\sqrt{r^2\!-\!|y|^2} \right)}{\partial y_j}\right )^2}\!dy\nonumber\\
&\!\!=&\!r\!\intl_{|y|< r} \!\!f\left(te_n \!+\!y, \sqrt{r^2\!-\!|y|^2}\right)\,\frac{dy}{\sqrt{r^2\!-\!|y|^2}}. \nonumber\eea
This gives (\ref{vt4}).
\end{proof}

\begin{remark} \label{rem1}{\rm By (\ref{vt4}), inversion of $V_+$ reduces to the similar problem for the Radon transform $R$. Most of the known inversion formulas for $R$  are applicable in the entire space $\rn$. From now on, when dealing with  $(R\vp) (\theta, t)$, we assume that
$\vp$  extends by zero outside $B_n$, so that $(V_+f)(\theta, t)=(R\vp) (\theta, t)=0$ whenever $|t|>1$.
}
\end{remark}
Lemmas \ref{lem2.1} and \ref{lem2.2} imply the following result.

\begin{theorem} \label {the1} A function $f\in L^1 (S^n_+)$ can be recovered from its vertical slice transform $V_+f$ by the formula
\be\label{vt5} f(x',x_{n+1})=x_{n+1}\, (R^{-1} \Phi)(x'), \quad \Phi(\theta, t)=\frac{1}{\sqrt {1\!-\!t^2}}\, (V_+f)(\theta, t),\ee
where $R^{-1}$ is the inverse Radon transform over hyperplanes in $\rn$.
\end{theorem}

There is  a big variety of explicit  formulas for $R^{-1}$; see, e.g.,  \cite{Ru15b} and references therein. For convenience of the reader, below we present some of them adapted to our case.

Let $h\equiv h(\theta, t)=\{x' \in \rn: x' \cdot \theta =t\}$ be a hyperplane in $\rn$ parametrized by $(\theta, t) \in S^{n-1} \times \bbr$.
The {\it dual Radon transform}  of the hyperplane function $\Phi (h) \equiv \Phi(\theta, t)$ is defined by the formula
\be\label{durt}
(R^*\Phi)(x')= \frac{1}{\sig_{n-1}}\intl_{S^{n-1}} \Phi(\theta, x'\cdot \theta)\,d\sig(\theta), \quad x'\in \rn, \ee
and  averages $\Phi$ over the set of all hyperplanes  passing through $x'$.

\subsection{ F. John's Inversion Method.}
 An advantage of this method in comparison with many others is that it can be easily adapted to functions on a ball. Below we present this reasoning following our exposition in \cite[p. 195]{Ru15b} that relies on the original work of F. John \cite [Chapter I]{J}.

The starting point  is the  decomposition of  $|x'|^{2-n}$ in spherical waves. Specifically,
\be  \label {llp9b} |x'|^{2-n}=\frac{a_1}{c_1}\, (-\Del)^{(n-1)/2} \intl_{S^{n-1}} |x'\cdot \theta|\, d\theta, \qquad n=3,5, \ldots\,,\ee
\be  \label {llp9bx}  |x'|^{2-n}=\frac{a_2}{c_2}\, (-\Del)^{(n-2)/2} \intl_{S^{n-1}} \log |x'\cdot \theta|\, d\theta, \qquad n=4,6, \ldots\,.\ee
Here $\Del$ is the Laplace operator in the $x'$-variable,
\[
a_1= \frac {\Gam ((n+1)/2)}{2\pi^{(n-1)/2}}, \qquad   c_1\!=\!-2^{n -2}\pi^{-1/2}\Gam\left (\frac{n\!+\!1}{2}\right )\,\Gam\left (\frac{n\!-\!2}{2}\right ),\]
\be  \label {llp9} a_2 =\frac {\Gam (n/2)}{2\pi^{n/2}},\qquad c_2\!=\!-2^{n -3}\pi^{-1/2}\Gam\left (\frac{n}{2}\right )\,\Gam\left (\frac{n\!-\!2}{2}\right ).\ee

We say that a function $\vp$ on a  set $G \subset \rn$ satisfies the  Lipschitz condition
 with exponent $0<\lam \le 1$ and write $\vp\in \Lip_\lam (G)$ if there is a  constant $A>0$ such that
\be\label{lippnew} |\vp(x')-\vp(y')| \le
A|x'-y'|^\lam \qquad \forall \,x',y' \in G.\ee

\begin{theorem}\label {the2} Let $\Phi=R\vp$, where $\vp\in \Lip_\lam (B_n)$ for some $0<\lam \le 1$.

\noindent {\rm (i)} If $n$ is odd, then
\be\label {i8v} \vp(x')=c_n (-\Del)^{(n-1)/2}(R^*\Phi)(x'),\qquad c_n=\frac{\pi^{1-n/2}}{2^{n-1}\,\Gam (n/2)}.\ee

\noindent {\rm (ii)} If  $n$ is even, then
\be\label {i8ve} \vp(x')= \tilde c_{n} \,(-\Del)^{n/2} (R^*L\Phi)(x'), \ee
\[\tilde c_{n} =\frac{\pi^{(1-n)/2}}{2^{n-1}\,\Gam (n/2)}, \qquad (L\Phi)(\theta,s)=\intl_{-1}^1 \log |s -t| \,\Phi(\theta,t)\, dt.\]
\end{theorem}

\begin{proof}
{\rm (i)}   We combine (\ref{llp9b}) with the classical equality $\vp (x')=-\Del I^{2} \vp(x')$, where
\be\label {hi8ve} (I^2 \vp) (x')=  \frac{1}{\gam_n (2)} \intl_{B_n}  \frac{\vp(y')\,dy}{|x'-y'|^{n-2}}, \quad \gam_n (2)=\frac{4\pi^{n/2}}{\Gamma((n-2)/2)},\ee
 is the Newtonian potential  and  $x'$ lies in the interior of $B_n$; see, e.g., \cite[p. 231]{Mi2}.  Let $\lam_1=a_1/c_1\, \gam_n (2)$. Then, using Remark \ref {rem1}, we have
\bea \vp(x')&=& \lam_1 \,(-\Del)^{(n+1)/2} \intl_{B_n}\vp(y')\, dy\intl_{S^{n-1}} |(x'-y')\cdot \theta|\, d\theta\nonumber\\
&=& \lam_1 \,(-\Del)^{(n+1)/2}\intl_{S^{n-1}} d\theta\intl_{-\infty}^\infty |t|\, (R_\theta \vp)(x'\cdot \theta -t)\, dt\nonumber\\
&=& \lam_1 \,(-\Del)^{(n+1)/2}  \intl_{S^{n-1}} d\theta\intl_{-1}^1 |x'\cdot \theta -t| \,(R_\theta \vp)(t)\, dt.\nonumber\eea
Observing that
\bea
&&\Del \intl_{-1}^1 |x'\cdot \theta -t| \,(R_\theta \vp)(t)\, dt\nonumber\\
&&=\frac{d^2}{dz^2}\,\Big [\intl_{-1}^z (z-t)\,(R_\theta \vp)(t)\, dt+\intl_{z}^1 (t-z)\,(R_\theta \vp)(t)\, dt\Big ]_{z=x'\cdot \theta }=2(R_\theta \vp)(x'\cdot \theta)\nonumber\eea
(here we use an obvious formula $\Del_{x'}' [g (x'\cdot \theta)]=g''(x'\cdot \theta)$),  we get
\[\vp(x')=-2\lam_1 (-\Del)^{(n-1)/2} \ \intl_{S^{n-1}}(R_\theta \vp)(x'\cdot \theta)\, d\theta.\]
This coincides with (\ref{i8v}).

\noindent {\rm (ii)} For $n$  even, we  start with (\ref{llp9bx}) if $n>2$, and with the  formula
\be \label {ml34}  \log |x'|=\frac {1}{2\pi} \!\intl_{S^{1}}\!  \log |x'\cdot \theta|\, d\theta -\frac {1}{\pi^{1/2}}
\,\frac{d}{d\a}\frac{\Gam (\a/2)}{\Gam ((1+\a)/2)}\Bigg |_{\a=1}  , \ee
  if $n=2$ (cf. \cite[formula (4.8.70)]{Ru15b}). Let \[\lam_2=\frac{a_2}{c_2\, \gam_n (2)}=\frac{1}{2^n\, \pi^{n-1/2}}.\] Then
\bea \vp(x')&=& \lam_2 \,(-\Del)^{n/2} \intl_{\bbr^n}\vp(y')\, dy'\intl_{S^{n-1}} \log |(x'-y')\cdot \theta|\, d\theta\nonumber\\
&=& \lam_2 \,(-\Del)^{n/2} \intl_{S^{n-1}} d\theta\intl_{-\infty}^\infty \log |t|\, (R_\theta \vp)(x'\cdot \theta -t)\, dt\nonumber\\
&=& \lam_2 \,(-\Del)^{n/2}   \intl_{S^{n-1}} d\theta\intl_{-1}^1 \log |x'\cdot \theta -t| \,(R_\theta \vp)(t)\, dt\nonumber\\
&=& \frac{\pi^{(1-n)/2}}{2^{n-1}\Gam (n/2)}\,(-\Del)^{n/2} (R^*LR\vp)(x').\nonumber\eea
\end{proof}

Theorem \ref{the2} gives the following result for $V_+f$.
\begin{theorem}\label {the3} Given an even function $f$ on $S^n$, suppose that
\be\label{cond1}
\frac{f(x', \sqrt{1-|x'|^2})}{\sqrt{1-|x'|^2}} \in \Lip_\lam (B_n)\quad \text {for some $0<\lam \le 1$}.\ee
Then $f$ can be recovered from its vertical slice transform $V_+f$ by the formula  $f(x',x_{n+1})=|x_{n+1}|\,\vp(x')$,
where $\vp$ is defined by  (\ref{i8v}) and (\ref{i8ve}) with $\Phi(\theta, t)=(1\!-\!t^2)^{-1/2} (V_+f)(\theta, t)$.

\end{theorem}

\subsection{Hypersingular integrals.}

Hypersingular integrals have proved to be a powerful tool for explicit inversion of operators of the potential type \cite{Ru96, Sa5}.
 Here we  follow \cite[Theorem 4.87]{Ru15b}.  Let
\[(\Del^\ell_{y'} \vp)(x')=\sum_{k=0}^\ell (-1)^k {\ell \choose k}
\vp(x'-ky')\]  be the finite difference of $\vp$ of order $\ell$
with step $y'$ at the point $x'$. We  denote
\[ c_n=2^{n-1}\pi^{n/2-1} \Gam (n/2), \qquad B_l (\a)=\sum_{k=0}^l (-1)^k {l \choose
k} k^{ \a},\]
\be\label{pech} d_{n,\ell}(\a)=\frac{\pi^{n/2}}{2^\a \Gamma ((n+\a)/2)} \left\{ \!
 \begin{array} {ll} \! \Gamma (-\a/2) B_l (\a), \!  & \mbox{if $ \a \neq 2, 4, 6, \ldots $,}\\
{}\\
\displaystyle{\frac{2 (-1)^{\a/2-1}}{(\a/2)!} \frac {d}{d\a} B_l (\a)}, & \mbox{if $ \a = 2, 4, 6, \ldots \, .$}\\
\end{array}
\right.\ee

The following result is a consequence of (\ref{vt5}) and \cite[Theorem 4.87]{Ru15b}.

\begin{theorem} \label {the4} Let $V_+f$ be the vertical slice transform of a function  $f\in L^1 (S^n_+)$ and let  $\Phi(\th, t)\!=\!(1-t^2)^{-1/2} (V_+f)(\th, t)$, as in (\ref{vt5}). Suppose that $\ell > n-1$ if $n$ is odd, and $\ell= n-1$
if $n$ is even. Then $f$ can be reconstructed by the formula

\be\label{pechka}
f(x',x_{n+1})=c\,x_{n+1}\intl_{\rn}
\frac{(\Del^\ell_{y'} R^*\Phi)(x')}{|y'|^{2n-1}}\, dy', \qquad c=\frac{c_n}{d_{n,\ell}(n-1)},\ee
where $\intl_{\rn} (...) =\lim\limits_{\e \to 0}\intl_{|y|>\e}(...)$ in the norm  of the space $L^1 (\rn)$  and in the a.e. sense.
 The constant $d_{n,\ell}(n-1)$ is defined by
 (\ref{pech})   with $\a=n-1$.
 If $f$ belongs to $C(S^n_+)$ and is supported away from the boundary, the limit is uniform on $\bbr^n$.
\end{theorem}

\begin {example} {\rm If $n=2$, then   \be
f(x',x_{n+1})=\frac{x_{n+1}}{4\pi}\intl_{\bbr^2}\frac{(R^*\Phi)(x')-(R^*\Phi)(y')}{|x'-y'|^3}\,
dy'.\ee
}
 \end{example}

\subsection{Singular Value Decomposition} By making use of the connection (\ref{vt4}), known results about  singular value decomposition of the Radon transform on the unit ball  can be adapted  for the vertical slice transform. Our reasoning relies on \cite[Subsection 4.4.4]{Ru15b}, where the corresponding results for  the Radon transform are presented in detail. Some information  for the non-weighted case $\vp \in L^2(B_n)$  can  be found in  \cite[pp. 17, 95] {Na1}.

Given a nonnegative measurable function $W$ on a measure space $X$, the corresponding weighted space $L^2 (X;W)$  is defined by
\[
L^2 (X;W)=\left \{ f: ||f||_{L^2 (X;W)}= \left (\intl_X |f(x)|^2 \,W(x)\, dx\right )^{1/2}<\infty \right \}.\]
We will be dealing with different weighted spaces of this kind. Let
\[ W(x')=(1-|x'|^2)^{n/2 -\lam}, \quad \tilde W(x)=x_{n+1}^{n-2\lam-1}; \quad  x'\in B_n, \quad  x\in S^n_+; \]
\[w(t)=(1-t^2)^{-1/2 -\lam},\quad \tilde  w(t)=(1-t^2)^{-1/2 -\lam}, \quad  t\in (-1,1);\]
\[\frC_n=\{(\th,t): \th \in S^{n-1},  \,t\in (-1,1)\};\]
\[
L^2_e (\frC_n;  w)=\{\Phi\in  L^2 (\frC_n;  w): \, \Phi(-\th, -t)= \Phi(\th, \t)\};\]
\[
L^2_e (\frC_n; \tilde w)=\{  F\in  L^2 (\frC_n; \tilde w):\,  F(-\th, -t)= F(\th, \t)\}.\]

Everywhere in the following we assume
\[\lam> n/2 -1.\]
This assumption is motivated by the following lemmas.
\begin{lemma} \cite[Example 4.36]{Ru15b} If $\lam> n/2 -1$, then the Radon transform $R$ is a linear bounded operator from $ L^2(B_{n}; W)$ to $L^2_e(\frC_n; w)$ with the norm
\be\label{cfa}  ||R||=\Bigg (\frac{\pi^{(n-1)/2}\, \Gam (\lam -n/2 +1)}{\Gam (\lam +1/2)}\Bigg )^{1/2}.\ee
\end{lemma}
  Note that the assumption $\lam> n/2 -1$ agrees with the argument of the gamma function $\Gam (\lam -n/2 +1)$.

A similar result for $V_+f$  follows by simple calculation.

\begin{lemma} \label {lo} If $\lam> n/2 -1$, then the vertical slice transform  $V_+$
 is a linear bounded operator from $ L^2(S^n_+; \tilde W)$ to $L^2_e(\frC_n; \tilde w)$ and $ ||V_+||=||R||$, as in (\ref{cfa}).
\end{lemma}
\begin{proof}
Let $\vp (x')=(1-|x'|^2)^{-1/2}\, f(x', (1-|x'|^2)^{1/2})$, as in (\ref{vt2}). Then (cf. the proof of Lemma \ref{lem2.1}),
\bea  ||\vp||^2_{L^2 (B_n;W)}&=&\intl_{B_n} |\vp(x')|^2 (1-|x'|^2)^{n/2 -\lam}\, dx'\nonumber\\
\label{vt4a1} &=&\intl_{S^n_+} |f(x)|^2 x_{n+1}^{n-2\lam-1}\, d\sig (x)= ||f||^2_{L^2 (S^n_+; \tilde W)}.\eea
Further, by (\ref {vt4}), $(R\vp) (\theta, t) = (V_+f)(\theta, t)(1-t^2)^{-1/2}$. Therefore,
\bea
&&||R\vp||^2_{L^2_e(\frC_n; w)}=\intl_{\frC_n} |(R\vp) (\theta, t)|^2 \, (1-t^2)^{1/2 -\lam} \, d\th dt\nonumber\\
\label{vt4a2}&&\intl_{\frC_n} |(V_+f)(\theta, t)|^2 \, (1-t^2)^{-1/2 -\lam} \, d\th dt=||V_+f||^2_{L^2_e(\frC_n; \tilde w)}.\eea
Because $||R\vp||^2_{L^2_e(\frC_n; w)}\le ||R||\, ||\vp||^2_{L^2 (B_n;W)}$, the result follows.
\end{proof}

The equalities (\ref{vt4a1}) and (\ref{vt4a2}) yield the following statements.

\begin{lemma} \label {lo1} ${}$

\noindent {\rm (i)} The maps
\[ \a:  L^2 (B_n;W) \to  L^2 (S^n_+;\tilde W),  \qquad  (\a \vp)(x) =x_{n+1} \vp(x'),\]
and
\[ \b: L^2_e (\frC_n; w) \to L^2_e (\frC_n; \tilde w), \qquad  (\b \Phi) (\th, t)= (1-t^2)^{1/2}\, \Phi (\theta, t),  \]
are isometric isomorphisms, so that
\be\label{kpp}  V_+f =\b R \a^{-1} f.\ee

\noindent {\rm (ii)} A system of functions $\{\eta_\nu\}_{\nu \in \N}$ is an orthonormal basis of $L^2 (B_n;W)$ if and only if
 a system  $\{\tilde\eta_\nu\}_{\nu \in \N}$ with $\tilde \eta_\nu (x)= x_{n+1}\eta_\nu (x')$ is an orthonormal basis of $L^2 (S^n_+;\tilde W)$.

\noindent {\rm (iii)} A system of functions $\{\zeta_\nu \}_{\nu \in \N}$ is an orthonormal basis of $L^2_e (\frC_n; w)$ if and only if
 a system  $\{\tilde\zeta_\nu\}_{\nu \in \N}$ with $\tilde \zeta_\nu (\th, t)= (1-t^2)^{1/2}\zeta_\nu (\th, t)$ is an orthonormal basis of $L^2_e (\frC_n; \tilde w)$.
\end{lemma}

We remind   explicit formulas for  $\eta_\nu$ and $\zeta_\nu$ following \cite[Subsection 4.12.3]{Ru15b}.
 Let $\{Y_{m, \mu}(\theta)\}$  be a real-valued orthonormal basis of spherical harmonics in
$L^2 (S^{n-1})$.  Here
 $m=0,1,2, \ldots$ and $ \mu=1,2, \ldots d_n (m)$, where
\be\label{kWSQRT} d_n(m) =(n+2m-2)\,
\frac{(n+m-3)!}{m! \, (n-2)!}\ee
is the dimension of the subspace of spherical harmonics of degree $m$.
 We denote
\be p_{2k}(x')=P_k^{(\lam-n/2,m+n/2 -1)}(2|x'|^2 -1),\qquad x'\in B_n,\ee
where $P_k^{(\lam-n/2,m+n/2 -1)}(t)$ is the Jacobi polynomial \cite{Er}, and introduce the index set
\be \label{hva4cu1}\N=\{\nu=(m,\mu,k):\, m,k=0,1,2, \ldots;   \; \mu=1,2, \ldots d_n (m)\}.\ee
The notation $C_m^\lam (t)$ is commonly used for Gegenbauer polynomials \cite{Er}.

\begin{lemma}\label{nvnvnv} \cite[pp. 237, 239]{Ru15b}  Let $\nu=(m,\mu,k)\in \N$.
The  functions
\bea \label{hva4cu2} \eta_\nu (x')&=&c_\nu \,|x'|^m  \,W^{-1}(x')\, p_{2k}(x')\,Y_{m, \mu} (x'/|x'|),\\
\label{hva4cu3} \zeta_\nu (\theta, t)\!&=&\!d_\nu  \,w^{-1}(t) \,C_{m+2k}^{\lam}\,(t)\,Y_{m, \mu} (\theta),\eea
with
\bea\label {56v8a}
c_\nu &=&\Bigg (\frac{2\,k!\, (2k+\lam +m)\, \Gam(k+m+\lam)}{ \Gam(k+\lam -n/2 +1)\,\Gam(k+m+n/2)} \Bigg )^{1/2},\\
\label {a56v8a} d_\nu &=&2^{\lam-1/2} \Gam (\lam) \left (\frac{(m+2k) ! \,(m+2k +\lam)}{\pi\, \Gam (m+2k+2\lam)}\right )^{1/2},\eea
 form  orthonormal bases of
$L^2 (B_n;W)$ and $L^2_e (\frC_n; w)$, respectively.
\end{lemma}

\begin{corollary} The functions \be\label {KLO}
\tilde \eta_\nu (x)= x_{n+1}\eta_\nu (x'), \qquad \tilde \zeta_\nu (\th, t)= (1-t^2)^{1/2}\zeta_\nu (\th, t),\ee
 where $\eta_\nu$ and $\zeta_\nu $ are defined  by (\ref{hva4cu2}) and (\ref{hva4cu3}),  form  orthonormal bases of $L^2 (S^n_+;\tilde W)$   and $L^2_e (\frC_n; \tilde w)$, respectively.
\end{corollary}

\begin{corollary} \label{xv}  The number
\be\label{98frdv}    s_\nu\!=\!2^{\lam}\,\pi^{(n-1)/2}\,
\left [\frac{(m\!+\!2k)!\,\Gam (k\!+\!m\!+\!\lam)\,\Gam (k\!+\!1\!+\!\lam\!-\!n/2) }{k!\,\Gam (m\!+\!2k\!+\!2\lam)\,\Gam (k\!+\!m\!+\!n/2)}\right]^{1/2} \ee
is  the  singular value of the vertical slice transform $V_+$ with respect to the
orthonormal systems   $\{\tilde\eta_\nu\}$  and $\{\tilde\zeta_\nu\}$, that is,
 $V_+\tilde\eta_\nu=s_\nu  \tilde\zeta_\nu$.
\end{corollary}
\begin{proof} By Lemma 4.124 from  \cite{Ru15b}, $R\eta_\nu =s_\nu \,\z_\nu$. Hence, by (\ref {kpp}) and (\ref {KLO}), $\b R \a^{-1} \tilde \eta_\nu = s_\nu \b \z_\nu=s_\nu\tilde \z_\nu$, that is,  $V_+\tilde \eta_\nu = s_\nu \tilde \z_\nu$.
\end{proof}

  The singular value decompositions of $V_+$  and its inverse are  consequences of Theorem 4.125
  from \cite{Ru15b}, Lemma \ref{lo1}, and Corollary \ref{xv}.  We set $ F(\theta, t)= V_+f(\theta, t)$ and use the following notation for the corresponding Fourier coefficients:
 \[
 f_\nu =\intl_{S^n_+} f(x) \,\tilde \eta_\nu (x) \,\tilde W(x)\, d\sig(x), \quad F_\nu= \intl_{\frC_n}  F(\th, t)\, \tilde \zeta_\nu (\th, t)\, 
 \tilde  w(t)\, d\th dt.\]

\begin{theorem}  \label{kkqqsw} Let $f\in L^2 (S^n_+; \tilde W)$. ${}$

\noindent {\rm (i)} The singular value decomposition of the vertical slice transform $V_+f$
 has the form
\be\label{98frdv1}
(V_+f)(\theta, t)=\sum\limits_{\nu} s_\nu \,f_\nu\,\tilde \zeta_\nu(\theta, t),\ee
where $s_\nu$ has the form  (\ref{98frdv}) and  $\tilde \zeta_\nu$ is defined by (\ref{KLO}) and (\ref{hva4cu3}).

\noindent {\rm (iii)}   The function $f$ can be reconstructed from $F=V_+f$ by the formula
\be\label{98frdv2}
f(x)=\sum\limits_{\nu} s_\nu^{-1} \,F_\nu\,\tilde\eta_\nu(x),
\ee
where $\tilde \eta_\nu$ is defined by  (\ref{KLO}) and (\ref{hva4cu2}).

\noindent The series   (\ref{98frdv1}) and (\ref{98frdv2}) converge in the $L^2_e (\frC_n; \tilde w)$-norm and in the  $L^2(S^n_+; \tilde W)$-norm, respectively.
\end{theorem}

\section{Method of Analytic Continuation}
 Given  $x\in S^n$ and $t \in (-1,1)$, let
\be \label {77b}
(Mf)(x, t)=\frac{(1-t^2)^{(1-n)/2}}{\sig_{n-1}}\intl_{x\cdot y=t} f (y)\, d\sig (y) \ee
be the mean value of  $f$  over the planar section  $\{y\in S^n: x\cdot y=t\}$. This operator is  commonly used  in  analysis on the sphere; see, e.g., \cite{Ru15b}  and references therein. If $x$ lies on the equator $S^{n-1}$, then  (\ref{77b}) is exactly our vertical slice transform, so that
\be \label {b77b}
(Vf)(\th, t)= \sig_{n-1}\, (Mf)(\th, t)\, (1-t^2)^{(n-1)/2}.\ee
Thus reconstruction of $f$ from $(Vf)(\th, t)$ is equivalent to reconstruction of $f$ from the spherical mean $Mf$ over geodesic spheres centered on the equator. It means that we can invoke the  method which was suggested in our previous paper \cite{AER}, of course, with suitable modifications.

Unlike the methods of the previous section, we can apply the method of analytic continuation only to infinitely differentiable  functions  which are even in the $x_{n+1}$ variable and vanish identically in some neighborhood of the equator $x_{n+1}=0$. The space of all  such functions will be denoted by
$\tilde C_e^\infty (S^n)$.
The subspace of integrable functions on $S^n$ which are even in the $x_{n+1}$-variable will be denoted by $L^1_e (S^n)$. The abbreviation $a.c.$ means analytic continuation.

It is convenient to treat the cases $n>2$ and $n=2$   separately.

\subsection{The case $n>2$} \label {98m}  We introduce an  analytic family of operators
\be
(N^\a f)(\th, t)\!=\!\intl_{S^n}  \!\frac{|\th\cdot y \!-\!t|^{\a -1}}{\Gam (\a/2)} f(y)\,  d\sig (y),\quad (\th, t) \!\in  \!\frC_n, \; Re \, \a \!> \!0,\ee
and a  backprojection operator $P$  that sends  functions on $\frC_n$ to  functions on $S^n$ by the formula
 \be\label{pfns}
(PF)(x)= \frac{1}{\sig_{n-1}}\,\intl_{S^{n-1}} F(\th, \th \cdot x)\,
d\sig (\th), \qquad x \in S^n.\ee
For $x$ and $y$ in $S^n$ we keep the previous notation $x=(x',x_{n+1})$, $y=(y',y_{n+1})$, where $x'$ and $y'$ are points in $B_n$.

\begin{lemma}\label{aaus}  If  $f\in L^1_e (S^n)$,
\be\label {mzs}
\vp(y')=(1-|y'|^2)^{-1/2}\,f(y', (1-|y'|^2)^{1/2}),
\ee
then
\be\label {anals} \underset
{\a=3-n}{a.c.} (P N^\a f)(x)=\frac{2\Gam (n/2)}{\pi^{1/2}}\intl_{B_n}  \frac {\vp(y')\, dy'}{|x'-y'|^{n-2}}.\ee
\end{lemma}
\begin{proof} For $Re \, \a >0$, changing the order of
integration, we have
\[
(PN^\a f)(x)=\intl_{S^n} f(y)\, k_\a (x,y) d\sig (y), \]
where
\bea k_\a (x,y)\!&=&\!\frac{1}{\sig_{n-1}}\intl_{S^{n-1}}\!
\frac{|\th\cdot (x\!-\! y)|^{\a -1}}{\Gam (\a/2)}
 d\sig (\th)\!=\!\frac{1}{\sig_{n-1}}\intl_{S^{n-1}}\!
\frac{|\th\cdot (x'\!-\! y')|^{\a -1}}{\Gam (\a/2)}
 d\sig (\th)\nonumber\\
 &=&\frac{|x'- y'|^{\a -1}}{\sig_{n-1}}\, \intl_{S^{n-1}}
\frac{|\th\cdot \om|^{\a -1}}{\Gam (\a/2)}
 d\sig (\th), \quad \om =\frac{x'- y'}{|x'- y'|}\in S^{n-1}.\nonumber\eea
The last integral can be evaluated by the formula
\[ \label{hva37e} \intl_{S^{n-1}}|\theta \cdot \eta|^{\a -1}\,
d\th=\frac{2\pi^{(n-1)/2}\Gam (\a/2)}{\Gam ((n+\a-1)/2)};\]
see, e.g., \cite [formula (1.12.14)]{Ru15b}. This gives
\be\label {vfr}
k_\a (x,y)=c_{n,\a}\,  |x'- y'|^{\a -1},\quad c_{n,\a}=\frac{\Gam (n/2)\, \pi^{-1/2}}{\Gam ((n+\a-1)/2)},  \ee
and therefore, by Lemma \ref{lem2.1},
\[(PN^\a f)(x)\!=\!c_{n,\a} \intl_{S^n}\! f(y)  |x'\!- \!y'|^{\a -1} \, d\sig (y)\!=\!2c_{n,\a} \intl_{B_n}\!  \vp(y')|x'\!- \!y'|^{\a -1} \, dy',\]
where $\vp \in L^1(B_n)$ has the form (\ref{mzs}). The last integral is an analytic function of $\a$ in the domain $Re \, \a> 1-n$ because one  can differentiate in $\a$ under the sign of integration. Taking analytic continuation at $\a=3-n$, we complete the proof.
\end{proof}

We will also need an alternative  representation of  $\underset
{\a=3-n}{a.c.} (PN^\a
f)(x)$ in terms of the spherical means (\ref{77b}).

\begin{lemma}\label{aauss}    Let $f\in \tilde C_e^\infty (S^n)$,
\be \label {777}\del_{n}=
\frac{(-1)^{[n/2-1]}\, \Gam((n-1)/2)}{(n\!-\!3)!}.\ee

\noindent {\rm (i)} If  $n=3,5, \ldots\, $, then
\be\label {anxs} \underset
{\a=3-n}{a.c.} (PN^\a
f)(x)=\del_{n}\,  \intl_{S^{n-1}} (d/dt)^{n-3} [(Mf)(\th, t)\, (1-t^2)^{n/2 -1}]\Big |_{t=\th\cdot x}\, d\th,\nonumber\ee

\noindent {\rm (ii)} If $n=4,6,\ldots\,$, then

\bea\label {anx2s} &&\underset
{\a=3-n}{a.c.} (PN^\a
f)(x)=-\frac{\del_{n}}{\pi} \,   \intl_{S^{n-1}}
d\sig (\th)\nonumber\\
&&\times\intl_{-1}^1  (d/dt)^{n-2} [(Mf)(\th, t)\, (1-t^2)^{n/2 -1}]\,\log |t-\th\cdot x|\,dt.\nonumber\eea
\end{lemma}

\begin{proof} For $Re \, \a >0$, by making use of the formula
\be\label{8ghv}
\intl_{S^n} f(y)\, a (\th\cdot y) \, dy=\sig_{n-1}\intl_{-1}^1 a(s) (Mf)(\th, s)\,(1-s^2)^{n/2 -1}\, ds\ee
 (see, e.g., \cite[formula (A.11.18]{Ru15b}),  we have
\bea
&&(N^\a f)(\th, t)=\frac{\sig_{n-1}}{\Gam (\a/2)}\intl_{-1}^1  (Mf)(\th, s)\,|s -t|^{\a -1} (1-s^2)^{n/2 -1}\, ds\nonumber\\
&&= \intl_{-\infty}^\infty \frac{|s|^{\a -1}}{\Gam ( \a/2)} \,  h_\th (s+t)\, dst, \quad  h_\th (t)=\sig_{n-1}\,(Mf)(\th, t) \,(1-t^2)^{n/2 -1}_+.\nonumber\eea
Here $(1-t^2)^{n/2 -1}_+$ denotes extension of $(1-t^2)^{n/2 -1}$  by zero outside $(-1,1)$.
Because $f$ is smooth and its support is separated from the equator,  $h_\th$ belongs to $ C^\infty (\bbr)$  uniformly in $\th$  and vanishes identically in the respective neighborhoods of $t=\pm 1$. Thus, we can invoke the standard procedure of analytic continuation (see, e.g., \cite{GSh1}, \cite[Lemma 2.1]{AER}),  and obtain the following equalities.

\noindent For $n=3,5, \ldots\, $:
$$
 \underset
{\a=3-n} {a.c.}\,(N^\a f)(\th, t)\!=\!\del_{n} \,  h_\th^{(n-3)} (t).$$

\noindent For $n=4,6,\ldots\,$:
$$
 \underset
{\a=3-n}{a.c.}\,(N^\a f)(\th, t)=-\frac{\del_{n}} {\pi}\,\intl_{-1}^1\!\!h_\th^{(n-2)} (s)\,\log |s\!-\!t|\, ds,$$
$\del_n$ being defined by (\ref{777}).

Combining these formulas with the backprojection $P$ and noting that operations
$ a.c.$ and $P$ commute,  we obtain
\[
(PN^\a f)(x)=\frac{\del_{n}}{\sig_{n-1}}\,\intl_{S^{n-1}}  h_\th^{(n-3)} (\th \cdot x)\, d\sig (\th),\]
if $n=3,5, \ldots\, $, and
\[
(PN^\a f)(x)=-\frac{\del_{n}}{\pi\,\sig_{n-1}}\,\intl_{S^{n-1}} d\sig (\th) \intl_{-1}^1  h_\th^{(n-2)} (t)\, \log |t-\th \cdot x|\, dt,\]
if $n=4,6,\ldots\,$. This gives the result.
\end{proof}

Now we compare different expressions of   $\underset
{\a=3-n} {a.c.}PN^\a f$ in  Lemmas   \ref{aaus}  and   \ref{aauss}. The right-hand side of (\ref{anals}) is $c_n\,I^2\vp$, where  $I^2\vp$  is the Newtonian potential (\ref{hi8ve}) and $c_n=4\pi^{(n-1)/2}(n-2)$. This gives the following corollary.

 \begin{corollary}  Let $f\in \tilde C_e^\infty (S^n)$, $n>2$.

\noindent {\rm (i)} If $n=3,5, \ldots\, $, then
\[
(I^2\vp)(x')=  \lam_n \intl_{S^{n-1}} (d/dt)^{n-3} [(Vf)(\th, t)\, (1-t^2)^{-1/2}]\Big |_{t=\th\cdot x'}\, d\th,\]
 \be\label{kio}  \lam_n= \frac {\del_{n}}{c_n\, \sig_{n-1}}=\frac{(-1)^{[n/2-1]}\, \Gam((n-1)/2)\, \Gam(n/2)}{8\pi^{n-1/2}\,(n\!-\!2)!}.\ee

 \noindent {\rm (ii)} If $n=4,6,\ldots\,$, then
\[
(I^2\vp)(x')= -\frac{\lam_n}{\pi} \intl_{S^{n-1}} d\sig (\th) \intl_{-1}^1 (d/dt)^{n-2} [(Vf)(\th, t)\, (1-t^2)^{-1/2}]\, \log |t-\th \cdot x'|\, dt.\]
\end{corollary}

Now, inverting $I^2\vp$  by making use of the Laplace operator $\Del=\partial_1^2 +\ldots +\partial_n^2$, we can reconstruct $\vp$, and therefore $f$.

\begin{theorem} \label{774}  If $n>2$, then a function $f\in \tilde C_e^\infty (S^n)$ can be reconstructed from the vertical slice transform $Vf$ as follows.

\noindent {\rm (i)} If $n=3,5, \ldots\, $, then
\[
f(x)=  \lam_n |x_{n+1}|  (-\Del) \intl_{S^{n-1}} (d/dt)^{n-3} [(Vf)(\th, t)\, (1-t^2)^{-1/2}]\Big |_{t=\th\cdot x'}\, d\th,\]
where $\lam_n$ is the constant (\ref{kio}).

 \noindent {\rm (ii)} If $n=4,6,\ldots\,$, then
\bea
f(x)\!\!&=&\! \frac{\lam_n}{\pi}\, |x_{n+1}| \, \Del \intl_{S^{n-1}} d\sig (\th) \nonumber\\
\label {87o} \!\!&\times&\!\!\intl_{-1}^1 (d/dt)^{n-2} [(Vf)(\th, t)\, (1\!-\!t^2)^{-1/2}]\, \log |t\!-\!\th \cdot x'|\, dt.\qquad \eea
\end{theorem}

\subsection{The case $n=2$}\label{5477}
In this case a substitute of the Newtonian potential $I^2\vp$  is the  logarithmic potential
\be\label {87f}
(I^2_{\ast}\vp)(x')=\frac{1}{2\pi}\intl_{B_2} \vp(y')\log|x'-y'|\,dy', \qquad x' \in B_2. \ee

\begin{lemma}\label{nn7s} If $f\in \tilde C_e^\infty (S^2)$, $\vp(x')\!=\!(1\!-\!|x'|^2)^{-1/2}\,f(x', (1\!-\!|x'|^2)^{1/2})$, then
\bea
(I^2_{\ast}\vp)(x')\!&=&\!\frac{1}{4\pi}\intl_{S^1} d\sig(\th) \intl_{-1}^1 (Mf)(\th, s) \log |s -\th \cdot x'|\, ds\nonumber\\
\label{nhy} &+&\frac{\log 2}{4\pi} \intl_{S^2} \!f(y)\, d\sig (y).\eea
\end{lemma}
\begin{proof}  Let
$$
\left( N_{\ast}f\right)(\th, t) =\intl_{S^2} f(y)\log|\th\cdot y-t|\,d\sig(y), \qquad (\th, t)\in S^1 \times  (-1,1),
$$
\be\label{pfnsx}
(P_*F)(x)= \frac{1}{2\pi}\,\intl_{S^{1}} F(\th, \th \cdot x)\,
d\sig (\th), \qquad x \in S^2.\ee
Changing the order of
integration,   we obtain
$$ (P_*N_{\ast} f)(x)=\intl_{S^2} f(y)\,k_{\ast} (x,y)\, d\sig (y),
$$ where
 \bea k_{\ast} (x,y)\!&=& \!\frac{1}{2\pi}\,\intl_{S^1}  \log|\th\cdot (x\!-\!y)|\,d\sig (\th)\nonumber\\
  \!&=&\! \log|x'\!-\!y'|+\frac{1}{2\pi}\,\intl_{S^1}  \log|\th\cdot \om|\,d\sig (\th), \quad \om\!=\!\frac{x'\!-\!y'}{|x'\!-\!y'|}\in S^1.\nonumber\eea
The second term can be easily evaluated:
\[
\frac{1}{2\pi}\,\intl_{S^1}  \log|\th\cdot \om|\,d\sig (\th)=\frac{1}{\pi} \intl_{-1}^1 \frac{\log |t|}{\sqrt {1-t^2}}\, dt=-\log 2;\]
see, e.g., \cite[formula 4.241 (7)]{GRy}. Thus
\[   k_{\ast} (x,y)=\log|x'\!-\!y'| -\log 2,\]
and  we have
\[
(P_*N_{\ast} f)(x)\!=\! \intl_{S^2}\! f(y) \log|x'\!-\!y'|\, d\sig (y) \!-\!c_f, \quad c_f\!=\!\log 2  \intl_{S^2} \!f(y)\, d\sig (y).\]
By Lemma \ref{lem2.1} it follows that
\be\label{p8sx}
(P_*N_{\ast} f)(x)= 2\intl_{B_2}\! \vp(y') \log|x'\!-\!y'|\, dy'\!-\!c_f= 4\pi (I^2_{\ast}\vp)(x')\!-\!c_f.\ee
On the other hand, by (\ref{8ghv}),
\[ (N_{\ast} f)(\th,t)=2\pi  \intl_{-1}^1 (Mf)(\th, s) \log |s -t|\, ds\]
and
\be\label{m23}
(P_*N_{\ast} f)(x)=\intl_{S^1} d\sig(\th) \intl_{-1}^1 (Mf)(\th, s) \log |s -\th \cdot x|\, ds.
\ee
Comparing (\ref{m23})  with (\ref{p8sx}),  we obtain (\ref{nhy}).
\end{proof}

Lemma \ref{nn7s}  gives the following inversion result.
\begin{theorem}
\label{Theorem n2.1s2} A function $f\in \tilde C_e^\infty (S^2)$ can be recovered from its vertical slice transform $Vf$ by the formula
\begin{equation}\label {fins}
f(x)  =\frac{|x_3|}{8\pi^2}\, \Delta  \intl_{S^1}d\sig(\th) \intl_{-1}^{1}\!(Vf)(\th, t) \frac{\log |t -\th \cdot x'|}{\sqrt{1-t^2}}\, dt.
\end{equation}
\end{theorem}
\begin{proof}
By (\ref{nhy}),
\bea
(I^2_{\ast}\vp)(x')\!&=&\!\frac{1}{4\pi}\intl_{S^1} d\sig(\th) \intl_{-1}^1 (Mf)(\th, s) \log |s -\th \cdot x'|\, ds\nonumber\\
\label{kuks} &+&\frac{\log 2}{4\pi} \intl_{S^2} \!f(y)\, d\sig (y).\nonumber\eea
Hence, owing to (\ref{b77b}) and the equality $\Del I^2_{\ast}\vp=\vp$, we obtain
\bea
\vp(x')&\equiv&(1-|x'|^2)^{-1/2}\,f(x', (1-|x'|^2)^{1/2})\nonumber\\
&=& \frac{1}{8\pi^2} \, \Del \intl_{S^1} d\sig(\th) \intl_{-1}^1 (Vf)(\th, t) \frac{\log |t -\th \cdot x'|}{\sqrt{1-t^2}}\, dt.\nonumber\eea
The latter is equivalent to (\ref{fins}).
\end{proof}

Note that (\ref{fins})  can be formally obtained from  (\ref {87o})  if we set $n=2$.

\begin{remark} Results of this section can  be applied to  explicit solution of the inverse problem for the
 Euler-Poisson-Darboux  equation
\be\label {epd}
\tilde\square_\a u\equiv \Del_S u - u_{\om\om} - (n-1+2\a) \cot \om\, u_\om + \a (n-1+\a)u=0,
\ee
where $x\in S^n$ is the space variable,  $\om\in (0,\pi)$ is the time variable, and  $\Del_S$ is the  Beltrami-Laplace operator on $S^n$ acting in the $x$-variable. The particular case  $\a\!=\!(1\!-\!n)/2$ gives  the wave equation in spherical tomography.  The  Cauchy problem
 \[  \tilde\square_\a u=0, \qquad u(x,0) = f(x), \quad  u_\om (x, 0) = 0,\]
is well-known; see, e.g.,  \cite {O2},  \cite[p. 179]{Ru00a}, and references therein.

The corresponding inverse problem is to find the initial function $f$ and reconstruct  $u(x,\om)$    for all  $(x,\om) \in S^n \times (0,\pi)$  when the solution  $u(x,\om)$    is known only for $x$  restricted to the equator $S^{n-1}$.  The interested reader is referred to
\cite[Section 6]{AER}, where this problem is solved in a more general setting for geodesic spheres of arbitrary fixed radius $0<\th\le \pi/2$. \end{remark}

\end{document}